\documentclass[12pt,oneside]{amsart}
\usepackage{url} 
\usepackage{fullpage}
\usepackage{amssymb}
\usepackage{latexsym}
\usepackage{amsmath}
\usepackage{mathrsfs}
\usepackage[all]{xy}
\usepackage{enumerate}
\usepackage{amsthm}
\usepackage{psfrag}
\usepackage{ifthen}
\usepackage{comment}
\usepackage{caption}
\usepackage{todonotes}
\usepackage{algorithm}
\usepackage{algpseudocode}

\usepackage{tikz}
\usetikzlibrary{cd}

\newtheorem{theorem}{Theorem}[section]
\newtheorem{lemma}[theorem]{Lemma}
\newtheorem{metalemma}[theorem]{Meta Lemma}
\newtheorem{proposition}[theorem]{Proposition}
\newtheorem{corollary}[theorem]{Corollary}
\newtheorem{problem}[theorem]{Problem}

\newtheorem*{theorem**}{Theorem\theoremnum}
\newenvironment{theorem*}[1][]{%
  \edef\theoremnum{\if\relax\detokenize{#1}\relax\else~#1\fi}
  \begin{theorem**}
}{%
  \end{theorem**}
}  

\theoremstyle{definition}
\newtheorem{definition}[theorem]{Definition}

\newtheorem{example}[theorem]{Example}

\newtheoremstyle{case}{}{}{}{}{}{:}{ }{}
\theoremstyle{case}

\newcommand{\bel}[1]{\begin{equation}\label{#1}}
\newcommand{\ee}{\end{equation}}

\newcommand{\LBA}{\left\{\begin{array}}
\newcommand{\EAR}{\end{array}\right.}

\def\ovw{{\overline{w}}}

\def\CN{{\mathcal N}}

\def\WP{{\mathbf{WP}}}

\makeatletter
\def\blfootnote{\xdef\@thefnmark{}\@footnotetext}
\makeatother

\newcommand{\gp}[1]{{\left\langle #1 \right\rangle}}
\newcommand{\gpr}[2]{{\left\langle #1 \mid #2 \right\rangle}}

\newcommand{\Set}[2]{\left\{\, #1 \;\middle|\; #2 \,\right\}}

\def\MN{{\mathbb{N}}}
\def\MZ{{\mathbb{Z}}}

\DeclareMathOperator{\Aut}{{Aut}}
\DeclareMathOperator{\Core}{{Core}}
\DeclareMathOperator{\Sch}{{Sch}}

\DeclareMathOperator{\first}{{first}}
\DeclareMathOperator{\last}{{last}}

\DeclareMathOperator{\val}{val}

\DeclareMathOperator{\rootSLP}{root}

\title{HNN extensions of free groups with equal associated subgroups of finite index:\\ polynomial time word problem}
\author{Hanwen Shen, Alexander Ushakov}
\date{September 2025}

\begin{document}

\maketitle

\begin{abstract}
Let $G=F\ast_\varphi t$ be an HNN extension of a free group $F$
with two equal associated normal subgroups $H_1 = H_2$ of finite index.
We prove that the word problem in $G$ is decidable in polynomial time.
This result extends to the case where the subgroups $H_1=H_2$
are not normal, provided that the isomorphism $\varphi:H_1\to H_2$
satisfies an additional condition described in Section \ref{se:normalizable}.

\end{abstract}

\section{Introduction}
\label{se:Introduction}

The study of computational problems in the theory of groups began
in the early twentieth century.
Two central themes in this area are decidability and computational complexity, 
that together shape our understanding of which problems can be solved 
algorithmically and how efficiently.
In his 1911 work \cite{Dehn:1911}, M. Dehn introduced 
three fundamental decision problems: 
the word problem, the conjugacy problem, and the isomorphism problem, 
that have since been central to the field. A significant result 
concerning decidability was established in the 1960s when P. Novikov \cite{Novikov:1955} and W. Boone \cite{booneWordProblem1959} demonstrated 
the existence of finitely presented groups for which the word problems 
are undecidable.
Nevertheless, for many important classes of 
groups, such as automatic groups, finitely generated linear groups, 
and finitely presented residually free groups, the word problem remains decidable.

The 1940s marked the introduction of HNN extensions by G. Higman, B. Neumann, 
and H. Neumann \cite{higmanEmbeddingTheoremsGroups1949}, providing a powerful 
tool for group embeddings and for constructing groups with special
algorithmic properties, where the word problem is typically decidable. 
Subsequent research in the 1970s, notably by C. Miller et al
\cite{Miller1}, further explored the computational complexity of  
HNN extensions of free groups. This led to the construction of Miller's machine, 
a group exhibiting a decidable word problem but an undecidable conjugacy problem.

\subsection{HNN extensions}
\label{se:HNN}

Let $G=\gpr{X}{R}$ be a group, $H_1,H_2\le G$ and $\varphi:H_1\rightarrow H_2$ 
be a group isomorphism.
The \emph{HNN extension} of $G$ relative to $\varphi$ is the group
denoted by $G\ast_\varphi t$,  given by the following presentation:
$$
G\ast_\varphi t = \gpr{X,t}{R,~ t^{-1}ht = \varphi(h),~ h\in H_1}.
$$
It is easy to see that if $H_1 = \gp{h_1,\ldots,h_k}$, then
$$
G\ast_\varphi t=
\gp{X,t \mid R,~ t^{-1}h_1t = \varphi(h_1),\ldots,t^{-1}h_kt = \varphi(h_k)}.
$$
For the group $G\ast_\varphi t$
\begin{itemize}
\item
the group $G$ is called the \emph{base group},
\item
$t$ is called the \emph{stable letter},
    \item
$H_1$ and $H_2$ are called the \emph{associated subgroups}.
\end{itemize}
Elements of $G\ast_{\varphi} t$ can be defined as alternating sequences
of the form
\begin{equation}\label{eq:w}
w =
w_0 t^{\varepsilon_1}
w_1 \dots 
w_{k-1} t^{\varepsilon_k} w_k,
\end{equation}
where $w_0,\dots,w_k$ are group words over the alphabet $X$ of $G$,
called \emph{syllables},
and $\varepsilon_i = \pm 1$.
The number $k$ is called the \emph{syllable length} of $w$.

We say that $w$ is \emph{$t$-reduced} if it is reduced and
does not involve the following subwords:
\begin{itemize}
\item
$t^{-1} w_i t$, where $w_i\in H_1$;
\item
$t w_i t^{-1}$, where $w_i\in H_2$.
\end{itemize}
Otherwise, we say that $w$ is not $t$-reduced.
If $w$ is not $t$-reduced, then it can be simplified as follows:
\begin{itemize}
\item
$t^{-1} w_i t$, where $w_i\in H_1$, can be replaced with $\varphi(w_i)$;
\item
$t w_i t^{-1}$, where $w_i\in H_2$, can be replaced with $\varphi^{-1}(w_i)$.
\end{itemize}
These operations are called \emph{$t$-reductions} (or \emph{Britton reductions}).
They do not change the corresponding group element and decrease
the syllable length of $w$.
Hence, in finitely many steps one obtains an equivalent $t$-reduced word.

\begin{lemma}[Britton's lemma, \cite{brittonWordProblem1963}]
$w=1$ in $G\ast_\varphi t$ and $k\ge 1$
$\ \ \Rightarrow\ \ w$ is not $t$-reduced.
\end{lemma}

\begin{corollary}
If the membership problem for $H_1$ and $H_2$ is decidable,
$\varphi$ and $\varphi^{-1}$ are computable, and the word problem
for $G$ is decidable,
then the word problem for $G\ast_\varphi t$ is decidable.
\end{corollary}

Current state of knowledge regarding the computational properties of the word problem for HNN extensions of free groups can be summarized as follows.
\begin{itemize}
\item 
The word problem, when approached via Britton’s lemma \cite{brittonWordProblem1963}, has exponential-time complexity.
\item 
In the generic (typical) case, the conjugacy problem can be solved in 
polynomial time \cite{borovikGenericComplexityConjugacy2009, Weiss2015}.
\item 
For ascending HNN extensions (when one of the subgroups is the entire group $G$)
M. Lohrey \cite{lohreyComplexityWordProblems2023} established polynomial-time decidability using straight-line programs. 
\item 
N. Haubold and M. Lohrey \cite{hauboldCompressedWordProblems} also proved that the compressed word problem for an HNN-extension with A finite is polynomial time Turing-reducible to the compressed word problem for the base group H.
\item 
A special case with equal subgroups associated by the identity
isomorphism can be solved in polynomial time \cite{Waack:1990}.
\end{itemize}

The main computational challenge of Britton reduction is that a single 
reduction step can multiply the length of a word by a constant factor, 
potentially producing words of exponential length.
We address this issue by representing such exponentially long words using 
straight-line programs (reviewed in Section \ref{se:SLP}) that define 
\textbf{paths} in the subgroup graphs of 
$H_1$ and $H_2$ (reviewed in Section \ref{se:sbgp-graphs}).

\subsection{Our results}
\label{se:Contribution}

Throughout the paper $F$ denotes the free group $F(X)$ over a finite alphabet $X$
and $H$ denotes a finitely generated subgroup of $F$. 
The main contributions of this paper are summarized in the following theorems.

\begin{theorem*}[\ref{th:main-case1}]
Suppose that $H$ is a normal subgroup of $F$ of finite index
and let $\varphi:H\to H$ be an automorphism.
Then the word problem for the HNN extension $F\ast_\varphi t$
is decidable in polynomial time.
\end{theorem*}

Theorem \ref{th:main-case1} can be generalized to the case
where $H$ is a subgroup of $F$ of finite index and 
$\varphi$ can be restricted to an automorphism $\varphi:N\to N$
of a normal subgroup $N\unlhd F$ of finite index.
We call such $\varphi$ \emph{normalizable} in Section \ref{se:normalizable}.

\begin{theorem*}[\ref{th:main-case2}]
\label{thm:NonNormalNiceIsomorphismCase}
Suppose that $H$ is a subgroup of $F$ of finite index 
and let $\varphi:H\to H$ be a normalizable isomorphism.
Then the word problem for the HNN extension $F\ast_\varphi t$
is decidable in polynomial time.
\end{theorem*}

\subsection{Outline}

The paper is organized as follows. 
Section~\ref{se:Freegroups} introduces essential preliminaries of free groups and subgroup graphs.
In Section~\ref{se:SLP} we discuss the definition and basic properties 
of straight-line programs.
Section~\ref{se:case1} presents a polynomial-time algorithm for the word problem
in the HNN extension $F\ast_\varphi t$ 
with equal associated normal subgroups of finite index,
which establishes Theorem~\ref{th:main-case1}.
Section~\ref{se:normalizable} introduces the notion of a \emph{normalizable} 
isomorphism $\varphi \colon H \to H$ and presents a polynomial-time algorithm 
for the word problem in the HNN extension $F \ast_\varphi t$ 
with equal associated subgroups of finite index and normalizable $\varphi$, 
which establishes Theorem~\ref{th:main-case2}.

\subsection{Model of computation and internal data representation}
We assume that all computations are performed on a random access machine.
Data representation for words is discussed in Section \ref{se:data-words} 
and data representation for straight-line programs is discussed 
in Section \ref{se:data-slps}.

\section{Preliminaries: subgroup graphs}
\label{se:sbgp-graphs}

\subsection{Free groups and free monoids}
\label{se:Freegroups}

Recall that an alphabet $X=\{x_1,\dots,x_n\}$ is a set, 
whose elements are called \emph{symbols}.
For $x\in X$ define the symbol $x^{-1}$ called the \emph{inverse} of $x$,
define the set $X^{-}=\Set{x^{-1}}{x\in X}$,
and form a \emph{symmetrized} alphabet 
(group alphabet)
$X^{\pm}=X\cup X^{-}$.
We refer to elements of $X$ as \emph{positive letters} and 
elements of $X^{-}$ as \emph{negative letters}.
The operation ${}^{-1}$ defines an involution on the set $X^\pm$,
mapping each $x\in X$ to $x^{-1}\in X^{-1}$ and 
$x^{-1}\in X^{-1}$ back to $x\in X$.

A \emph{word} over the alphabet $X$ is a sequence of letters from $X$.
The empty sequence of letters (the \emph{empty word}) is denoted by $\varepsilon$.
In our notation for words, we omit commas between letters and simply
write $w=x_1\dots x_n$.
The set of all words over the alphabet $X$ is denoted by $X^\ast$.
The set $X^\ast$ equipped with the binary operation of concatenation
is a free monoid.

A \emph{group word} $w$ is a word over a group alphabet $X^\pm$.
We use the following notation for group words:
$$
w= x_1^{\varepsilon_1} \dots x_n^{\varepsilon_n}
$$
where $x_i\in X$ and $\varepsilon_i = \pm 1$.
We say that $w$ is \emph{reduced} if it does not contain any pair 
of consecutive inverse letters, that is, any subword of the form
$x x^{-1}$ or $x^{-1} x$.
Denote by $F(X)$ the set of all reduced group words over $X$.
Every word $w$ can be reduced by a process called \emph{reduction}
which successively removes occurrences of subwords of the form
$x x^{-1}$ or $x^{-1} x$
until no such subwords remain.
The result of reducing any word $w$
is uniquely defined, that is, it does not depend on a particular 
sequence of removals.
Denote by $\ovw$ the result of reducing $w$.
The set $F(X)$ equipped with the multiplication
operation $\cdot$ defined by
$$
u\cdot v = \overline{u\circ v}
$$
is a free group. 
In this paper we mainly consider group words, and for simplicity,
we refer to them as words.

\subsubsection{Data representation for words}
\label{se:data-words}

A positive letter $x_i$ of an alphabet $X=\{x_1,\dots,x_n\}$ is encoded
by $i\in\MZ$ and a negative letter $x_i^{-1}$ is encoded
by $-i\in\MZ$. A word $w=w(X)$ is encoded by a sequence of integers.

\subsection{Subgroup graph}

Here we review the definition of subgroup graphs and recall 
their basic properties.
We assume the reader is familiar with this material
and omit the proofs. All relevant proofs 
can be found in \cite{Kapovich_Miasnikov:2002}.

An \emph{$X$-digraph} $\Gamma$ is a tuple $(V,E^\pm,\mu,r)$, where 
\begin{itemize}
\item 
$(V,E^\pm)$ defines a directed graph,
\item 
$r\in V$ is a designated vertex, called the \emph{root},
\item 
$\mu:E^\pm\to X^{\pm}$ is an edge labeling function
(we often use notation $u\stackrel{x}{\to} v$ for an edge $e$
labeled with $\mu(e)=x\in X^\pm$ that starts at $u$ and leads to $v$).
\end{itemize}
Define 
$$E^+=\Set{e\in E^\pm}{\mu(e)\in X}
\ \mbox{ and }\ 
E^-=\Set{e\in E^\pm}{\mu(e)\in X^{-}}$$
called  the set of \emph{positive} and \emph{negative} edges respectively.
Clearly, $E^\pm = E^+ \sqcup E^-$.
We say that edges $e_1,e_2\in E^\pm$ are \emph{inverses} of each other if 
$$
e_1=u\stackrel{x}{\to} v
\ \mbox{ and }\ 
e_2=v\stackrel{x^{-1}}{\to} u,
$$
i.e., if they have the same endpoints, opposite direction,
and opposite labels,
in which case we write $e_2=e_1^{-1}$ and $e_1=e_2^{-1}$.
We say that the edges in $\Gamma$ are \emph{inversible}
if $\Gamma$ with every edge $e=u\stackrel{x}{\to} v$ 
contains its inverse. 
We say that $\Gamma$ is \emph{folded} if for every $v\in V$
and $x\in X^{\pm}$ there exists at most one edge starting from $v$
labeled with $x$.

For an edge $e=u\stackrel{x}{\to} v$ we denote 
its \emph{origin} $u$ by $o(e)$ and its \emph{terminus} $v$ by $t(e)$.
A \emph{path} $p$ in $\Gamma$ is a sequence of edges
$e_{1},\dots,e_{t}$ satisfying the following 
\emph{connectedness} condition:
$$
t(e_{s})=o(e_{s+1}),
$$
for every $s=1,\dots,t-1$.
The \emph{label} $\mu(p)$ of a path $p$ is the word
$$
\mu(p)=\mu(e_{1})\dots\mu(e_{t})\ \in\ (X^{\pm})^\ast.
$$
We say that $p$ is \emph{reduced} if it does not contain consecutive 
opposite edges $ee^{-1}$.
To reduce $p$ means to delete all pairs of consecutive 
opposite edges from $p$.
It is easy to show that the result of path-reduction is uniquely defined,
i.e., it does not depend on the sequence of reductions.

A \emph{circuit} in $\Gamma$ is a closed path from $r$ to $r$.
We say that $\Gamma$ is a \emph{core} graph 
if for every edge $e$ there exists a reduced circuit in $\Gamma$
containing $e$.
An $X$-digraph $\Gamma = (V,E^\pm,\mu,r)$ is called a \emph{subgroup graph} 
if it is a core graph, is folded and connected, and has inversible edges.

If $\Gamma$ is not folded, then there are distinct edges 
$e_1=v\stackrel{x}{\to} u_1$ and $e_2=v\stackrel{x}{\to} u_2$
with the same origin $v$ and the same label $x$.
Identifying the edges $e_1$ and $e_2$ (and vertices $u_1$ and $u_2$) 
defines a single folding step.
A sequence of foldings eventually terminates with a folded graph
because each folding step decreases the size of $\Gamma$.
It can be shown that the result does not depend on the specific 
sequence of foldings applied.
The folding can be performed in nearly linear time, see \cite{Touikan:2006}.

Folded graphs with inversible edges have the following important property: 
for any path $p$ we have
$$
p \mbox{ is a reduced path}
\ \ \Leftrightarrow\ \ 
\mu(p) \mbox{ is a reduced word.}
$$
We say that an $X$-digraph $\Gamma=(V,E^\pm,\mu,r)$
\emph{accepts} a word $w\in F(X)$ if
$\Gamma$ contains a path $p$ from $r$ to $r$ labeled with $w$.
The language of all accepted words is defined by
\begin{equation}\label{eq:L}
L(\Gamma) = L(\Gamma,r) = \Set{w\in F(X)}{\Gamma \mbox{ accepts } w}.
\end{equation}
It is easy to see that $L(\Gamma)$ is a subgroup of $F(X)$ when $\Gamma$ is a subgroup graph.

\subsection{A basis for $L(\Gamma)$}
\label{se:sbgp-basis}

Let $\Gamma=(V,E^\pm,\mu,r)$ be a subgroup graph.
In this section we outline a procedure for
finding a free basis for the subgroup $L(\Gamma)$.

Since $\Gamma$ is inversible, the set $E^+$ uniquely defines the set $E^-$.
Hence, we can regard each pair of edges $\{e,e^{-1}\}$
as a single edge traversable in both directions, reading the label 
$x$ going in one direction and $x^{-1}$ in the other. 
From this perspective
$(V,E^\pm)$ can be viewed as an undirected graph $(V,E)$, 
where the edges $E$ are uniquely defined by $E^+$. 
A path in $(V,E)$ is a sequence of edges 
$e_1,\dots,e_k$ from $E^+$,
where each edge is either traversed in the forward (direct) direction 
or in the inverse direction.

We say that $T\subseteq E^+$ defines a spanning tree in $\Gamma$
if $(V,T)$ is a tree as an undirected graph. For a vertex $v\in V$
let $[r,v]_T$ be the unique reduced path in $T$ from $r$ to $v$
and $\mu([r,v]_T)$ its label. 
For $e=u\stackrel{x}{\to} v\in E^+$ define the circuit 
\begin{equation}\label{eq:p_e}
p_e = [r,o(e)]_T \cdot e\cdot  [t(e),r]_T
\end{equation}
from $r$ to $r$ in $\Gamma$ and its label $w_e = \mu(p_e)$.
Clearly, $w_e=1$ if and only if $e\in T$.

\begin{proposition}[{{\cite[Lemma 6.1]{Kapovich_Miasnikov:2002}}}]
$L(\Gamma) = \gp{w_e \mid e\in E^+\setminus T}$.
\end{proposition}

\subsection{Schreier graph}

Recall that a \emph{right coset} of a subgroup $H\le G$ is the set
$$
Hg = \Set{hg}{h\in H}.
$$
The collection of right cosets forms a partition of $G$.
The number of distinct cosets of $H$ in $G$
is called the \emph{index} of $H$ in $G$, denoted by $|G:H|$.

Consider a subgroup $H\le G$ of a group $G$ generated by
$x_1,\dots,x_n\in G$.
The \emph{Schreier graph} of $H$ with respect to a generating set
$X=\{x_1,\dots,x_n\}$ 
is an $X$-digraph $\Sch(H,X)=(V,E,\mu,1_H)$ defined by
$$
V=\Set{Hg}{g\in G}
\mbox{ and }
E=\Set{Hg\stackrel{x}{\to}Hgx}{g\in G,\ x\in X^\pm},
$$
with the designated root $1_H=H\cdot 1\in V$, where $1$ is the
identity in $G$.
By construction, $\Sch(H,X)$ is 
\begin{itemize}
\item 
folded and connected;
\item 
has inversible edges;
\item 
in general, it is not a core graph;
\item 
$|V|=[G:H]$;
\item 
$L(\Gamma) = H$.
\end{itemize}

For an $X$-digraph $\Gamma$ and $v\in V(\Gamma)$ define
the core $\Core(\Gamma,v)$
of $\Gamma$ with respect to $v$ as the subgraph
induced by all reduced paths from $v$ to $v$ in $\Gamma$.
It is easy to see that
$\Gamma' = \Core(\Gamma,r)$, where $\Gamma=(V,E,\mu,r)$, is a core graph
defining the same subgroup, i.e., $L(\Gamma)=L(\Gamma')$.

\begin{theorem}[{{\cite[Theorem 5.1, Theorem 5.2, and Definition 5.3]{Kapovich_Miasnikov:2002}}}]
If $H$ is a subgroup of $F(X)$, then there is a unique 
(up to an isomorphism) subgroup graph $\Gamma$ satisfying 
$L(\Gamma) = H$.
Denote this graph by $\Gamma_H$.
\end{theorem}

\begin{proof}
In fact, $\Core(\Sch(H,X))$ is the required graph.
\end{proof}

\subsection{Subgroup graph homomorphism}

Let $\Gamma_i=(V_i,E_i,\mu_i,r_i)$ for $i=1,2$ be subgroup graphs.
Recall that a map $\varphi:V_1\to V_2$ 
is a \emph{subgroup graph homomorphism} if
\begin{itemize}
\item 
$\varphi(r_1)=r_2$;
\item 
$u\stackrel{x}{\to} v$ belongs to $E_1$
$\ \ \Leftrightarrow\ \ $
$\varphi(u)\stackrel{x}{\to} \varphi(v)$ belongs to $E_2$.
\end{itemize}

\begin{proposition}[{{\cite[Lemma 4.1 and Proposition 4.3]{Kapovich_Miasnikov:2002}}}]
$H_1\le H_2\ \Leftrightarrow\ $ there exists a subgroup graph homomorphism
$\varphi:\Gamma_{H_2}\to \Gamma_{H_1}$.
\end{proposition}

\subsection{Regularity, self-similarity, and shift operation}
\label{se:shift}

Here we introduce the shift operation on subgroup graphs 
and discuss two properties that allow it to be computed efficiently.
We say that a subgroup graph $\Gamma = (V,E^\pm,\mu,r)$
is \emph{$X$-regular} (or deterministic) 
if for each vertex $v$ of $\Gamma$
and for each $x\in X^{\pm}$, there is exactly one edge $e$ 
with $o(e)=v$ and $\mu(e)=x$.

\begin{proposition}[{{\cite[cf. Proposition 8.3]{Kapovich_Miasnikov:2002}}}]
Let $\Gamma$ be the subgroup graph of $H\le F$.
Then
$$
[F:H]=
\begin{cases}
|\Gamma| & \mbox{ if } \Gamma \mbox{ is $X$-regular,}\\
\infty & \mbox{ otherwise.}
\end{cases}
$$
In particular,
$[F:H]<\infty\ \ \Leftrightarrow\ \ \Gamma_H$
is finite and $X$-regular.
\end{proposition}

By $\Aut(\Gamma)$ we denote the \emph{group of automorphisms} of $\Gamma$.
We say that $\Gamma = (V,E^\pm,\mu,r)$ is \emph{self-similar} if
for every $u,v \in V$ there exists an automorphism of $\Gamma$ mapping $u$ to $v$; 
for a folded graph $\Gamma$ such an automorphism, when it exists, is unique.
We denote this automorphism by $S_{u,v}$ and refer to it as a \emph{shift operation}.
Note that $S_{u,v}$ induces 
\begin{itemize}
\item 
a permutation on the set of vertices $V$;
\item 
a permutation on the set of edges $E^\pm$;
\item 
a bijection from sequences of edges 
to sequences of edges
$$
e_1\dots e_k
\ \ \stackrel{S_{u,v}}{\mapsto} \ \ 
S_{u,v}(e_1) \dots S_{u,v}(e_k),
$$
and the corresponding bijection from 
the set of paths that start at the vertex $u$ to
the set of paths that start at $v$.
\end{itemize}
We use the same notation $S_{u,v}$ for the induced functions.
Clearly, shift operations preserve labels, i.e.,
for every $u,v$ and a sequence of edges $p$ we have
$$
\mu(S_{u,v}(p)) = \mu(p).
$$

\begin{theorem}[{{\cite[Theorem 8.14]{Kapovich_Miasnikov:2002}}}]
$H\unlhd F(X)$ if and only if
$\Gamma_H$ is $X$-regular and self-similar.
\end{theorem}

\section{Preliminaries: straight-line programs}
\label{se:SLP}

In this section we review the definition of straight-line programs,
following the exposition in \cite[Chapter 19]{MSU_book:2011}.
See also \cite{lohreyCompressedWordProblem2014} for further background.

\subsection{Definition of a straight-line program}

Formally, a \emph{straight-line program} (SLP) is a quadruple 
$P = (X, \CN, R, \delta)$, where
\begin{itemize}
    \item
$X = \{x_1, \ldots, x_n\}$ is a finite set 
of \emph{terminal symbols}
(the \emph{alphabet}).
    \item
$\CN$ is a finite set of {\em non-terminal symbols}.
    \item
$R \in \CN$ is the {\em root symbol}, also denoted by $\rootSLP(P)$.
    \item
$\delta : \CN \to X \cup \{\varepsilon\} \cup  (\CN\times \CN)$
is a \emph{production function} that determines the set of \emph{production rules},
where $\varepsilon$ is a special symbol that denotes
the \emph{empty word}.
There are two types of production rules defined by $\delta(N)$ for $N\in\CN$:
\begin{itemize}
\item 
$\delta(N) = x \in X \cup \{\varepsilon\}$, 
\item 
$\delta(N) = (A,B) \in \CN\times \CN$.
\end{itemize}
To be called an SLP, $P$ must define an acyclic production graph, defined below.
\end{itemize}

The \emph{production graph} for $P$ is a directed graph $G(P) = (V,E)$, 
where $V = X \sqcup \{\varepsilon\} \sqcup \CN$ and
\begin{align*}
E= 
& \ \ \  \{(N,\delta(N)) \mid \delta(N) \in X\cup\{\varepsilon\}\} \\
& \cup \{(N,A) \mid \delta(N) = (A,B) \mbox{ for some } B\in\CN\} \\
& \cup \{(N,B) \mid \delta(N) = (A,B) \mbox{ for some } A\in\CN\}.
\end{align*}
The graph $G(P)$ is \emph{acyclic} if it does not contain a directed cycle.

For an SLP $P = (X, \CN, R, \delta)$ inductively define a function 
$\val : \CN \rightarrow X^\ast$ by
$$
\val(N) =
\begin{cases}
x & \mbox{if } \delta(N) = x \in X \cup \{\varepsilon\},\\
\val(A) \val(B) & \mbox{if } \delta(N) = (A,B),\\
\end{cases}
$$
and the sequence $\val(P)$ as $\val(R).$
The word $\val(P)$ is called the \emph{output} of $P$.
If $X$ is a group alphabet and $\val(N)$ is a
reduced word for every $N\in \CN$, then we say that $P$ is 
\emph{reduced}.

In all cases considered in this paper the set of terminals $X$ is fixed.
Therefore, we define the size of an SLP $P$ as
the size of $\CN$, denoted by $|P|$.

We say that a non-terminal $N\in \CN$ in $P$ is \emph{essential} 
if the following two conditions hold:
\begin{itemize}
\item[(a)]
The graph $G(P)$ contains a path from the root $R$ to $N$.
Otherwise, the value of $N$ does not contribute to $\val(P)$,
and $N$ can be removed from $\CN$.
\item[(b)]
Either $\delta(N) \neq \varepsilon$ or $N = R$.
Otherwise, $N$ is a non-root intermediate non-terminal with
$\val(N) = \varepsilon$ and, hence, can be removed from $\CN$
(with an appropriate redefinition of $\delta$).
\end{itemize}
It is easy to see that all nonessential non-terminals can be removed in
$O(|P|)$ time; we refer to this process as 
the \emph{pruning procedure}.
Throughout the paper we assume that all non-terminals are essential.

\subsection{Data representation for SLPs}
\label{se:data-slps}

In all our computations the alphabet $X$ is fixed and all operations
on SLPs are actually performed on $\CN$ and $\delta$.
To simplify analysis, we make two assumptions.
\begin{itemize}
\item 
\textsc{(Assumption-I).}
We have a sufficiently large pool 
of symbols available for non-terminals
and that it takes $O(1)$ time to generate a symbol not involved in
any of the currently used SLPs.
\item 
\textsc{(Assumption-II).}
The function $\delta$ is stored in a container that enables
$O(1)$ time complexity for the following manipulations:
\begin{itemize}
\item 
for a given $N\in\CN$ get $\delta(N)$;
\item 
for a given $N\in\CN$ delete the production for $N$;
\item 
for a given $N\in\CN$ and $pr\in X \cup (\CN\times \CN)$
add the production $\delta(N)= pr$ to $\delta$;
\item 
for a given $N\in\CN$ modify the value of $\delta(N)$.
\end{itemize}
In particular, for two functions $\delta_1,\delta_2$ with disjoint supports
there is a procedure that 
adds the description of $\delta_2$ to the description of $\delta_1$
in $O(|\delta_2|)$ time.
\end{itemize}
We emphasize that Assumptions~I--II are made solely for convenience.
Any standard representation of SLPs (e.g., encoding non-terminals as natural
numbers and storing the production function $\delta$ as a dictionary)
would introduce at most a polylogarithmic overhead in the running time.
Since we are only concerned with polynomial-time computability,
these assumptions can be adopted without loss of generality.

\subsection{Basic properties}

Here we discuss basic computational properties of SLPs.

\begin{lemma}\label{le:w-trivial}
For a given SLP $P$ it takes $O(|P|)$ time to decide whether $\val(P)=\varepsilon$.
\end{lemma}

\begin{proof}
Clearly, for any $N\in\CN$ we have $\val(N)=\varepsilon$ if and only if
one of the following two conditions is satisfied:
\begin{itemize}
\item
$\delta(N) = \varepsilon$, or
\item
$\delta(N) = (A,B)$ and $\val(A)=\varepsilon$ and $\val(B)=\varepsilon$.
\end{itemize}
Hence, we can decide if $\val(N)=\varepsilon$ 
for all non-terminals $N\in\CN$ in linear time $O(|\CN|)$
by starting at the root 
$R$, proceeding recursively down the tree to the leaves, 
and then traversing back up to the root.
\end{proof}

For $N\in\CN$ denote by $\first(N)$ and $\last(N)$
the first and the last element in $\val(N)$ respectively, if $\val(N)\ne\varepsilon$.
If $\val(N)=\varepsilon$, then we write 
$\first(N)=\last(N)=\varnothing$.

\begin{lemma}\label{le:first-last}
Given an SLP $P=(X,\CN,R,\delta)$, it takes $O(|P|)$ time to compute
$\first(R)$ and $\last(R)$.
\end{lemma}

\begin{proof}
For every $N\in\CN$ the following holds:
\begin{itemize}
\item 
$\first(N)=\varnothing$ if $\delta(N) = \varepsilon$.
\item 
$\first(N)=x$ if $\delta(N)  = x \in X$.
\item 
$\first(N)=\first(A)$  if $\delta(N) = (A,B)$ and $\val(A)\ne\varepsilon$.
\item 
$\first(N)=\first(B)$  if $\delta(N) = (A,B)$ and $\val(A)=\varepsilon$.
\end{itemize}
We can use these formulae to compute
$\first(N)$ for all non-terminals $n\in\CN$ in linear time $O(|\CN|)$
as in Lemma \ref{le:w-trivial}.
The value of $\last(N)$ can be computed similarly.
\end{proof}

\begin{lemma}\label{le:CP-w}
For a given word $w = x_1 \dots x_k$, where $x_i\in X$,
it requires $O(|w|)$ time to construct 
an SLP $P_w$ satisfying $\val(P_w) = w$.
\end{lemma}

\begin{proof}
Clearly, the statement holds when $|w|=0$ or $|w|=1$.
Let $X'\subseteq X$ be the set of letters involved in $w$.
Let $\CN=X' \cup \{A_1,\dots,A_{k-1}\}$.
Define an SLP $P_w=(X,\CN,A_{k-1},\delta)$, where $\delta$
is defined as follows:
\begin{itemize}
\item 
$\delta(A_x)=x$ for $x\in X'$,
\item 
$\delta(A_1) = (A_{x_1}, A_{x_2})$,
\item 
$\delta(A_2) = (A_1, A_{x_3}), \dots,\delta(A_{k-1}) = (A_{k-2}, A_{x_k})$.
\end{itemize}
$P_w$ can be constructed in $O(|w|)$ time and
satisfies $\val(P_w) = w$.
\end{proof}

\subsection{SLP concatenation}

Consider two straight-line programs 
$P_1 = (X, \CN_1, R_1, \delta_1)$ and
$P_2 = (X, \CN_2, R_2, \delta_2)$ over the same alphabet $X$.
Assuming that $\CN_1\cap\CN_2 = \varnothing$ and $A\notin \CN_1\cup\CN_2$
define a new SLP $P = (X, \CN_1\cup\CN_2 \cup \{A\}, A, \delta)$, where
$\delta$ is defined by
$$
\delta(N) = 
\begin{cases}
\delta_1(N) & \mbox{ if } N\in\CN_1,\\
\delta_2(N) & \mbox{ if } N\in\CN_2,\\
(R_1,R_2) & \mbox{ if } N=A.
\end{cases}
$$

\begin{lemma}
$\val(P) = \val(P_1) \circ \val(P_2)$.
\end{lemma}

\begin{proof}
$\val(P) = \val(A) = \val(R_1)\circ \val(R_2) = \val(P_1)\circ \val(P_2)$.
\end{proof}

Denote the SLP $P$ by $P_1\circ P_2$. More generally,
for SLPs $P_1,\dots,P_k$ denote by
$P_1\circ\dots\circ P_k$ the SLP
$((P_1\circ P_2) \circ P_3)\dots\circ P_k$.
Notice that concatenating $k$ SLPs requires $k-1$
additional non-terminals.

\subsection{Straight-line program over an $X$-digraph}

Let $\Gamma=(V,E^\pm,\mu,r)$ be a subgroup graph over the alphabet $X$.
We can treat the set of edges $E^\pm$ as an alphabet.
Note that $E^\pm$ forms a group alphabet since, by assumption, 
$\Gamma$ contains with every edge $e$ its inverse $e^{-1}$.
Hence, we can work with SLPs over $E^\pm$.
The output $\val(P)$ of such SLP $P$ is a sequence of edges in $\Gamma$.

Let $P$ be an SLP over an $X$-digraph $\Gamma$.
For $N\in\CN$ define vertices $o(N)$ and $t(N)$ as
$$
o(N) = o(\first(N)) 
\mbox{ and }
t(N) = t(\last(N)), 
$$
if $\val(N)\ne\varepsilon$ and as $\varnothing$ if  
$\val(N) = \varepsilon$.

\begin{lemma}\label{le:SLP-connected}
It takes linear time to decide if the sequence of edges $\val(P)$
is a path.
\end{lemma}

\begin{proof}
For an SLP $P$ with all essential non-terminals, $\val(P)$ is a path if and only if 
$$
\forall N\in\CN
\ \ \ 
\delta(N)=(A,B)\ \wedge\ \val(A)\ne\varepsilon\ \wedge\ \val(B)\ne\varepsilon
\ \rightarrow\ 
t(A)=o(B).
$$
This condition can be checked in linear time because,
by Lemma \ref{le:first-last},
$t(A)$ and $o(B)$ can be computed in linear time for all 
non-terminals in $P$.
\end{proof}

In the next proposition we assume that $\Gamma$
is a \textbf{fixed} subgroup graph.
This allows us to treat all relevant data related to
$\Gamma$, such as an 
explicit description of the shift operation $S_{u,v}$, 
as precomputed.

\begin{proposition}\label{pr:Suv-SLP}
Let $\Gamma=(V,E^\pm,\mu,r)$ be an $X$-regular 
and self-similar subgroup graph.
Let $u,v\in V$ and $S_{u,v}:E^\pm\to E^\pm$ be a 
permutation on the set of edges given explicitly
as a set of pairs $(e,S_{u,v}(e))$.
Given an SLP $P$ over $\Gamma$ 
it requires $O(|P|)$ time to construct an SLP $P'$
satisfying
\begin{itemize}
\item 
$\val(P')=S_{u,v}(\val(P))$,
\item 
$|P'|=|P|$.
\end{itemize}
\end{proposition}

\begin{proof}
For every non-terminal $N$ such that $\delta(N)=e\in E^\pm$,
the procedure replaces $e$ with $S_{u,v}(e)$.
This does not change the number of non-terminals.
\end{proof}

Denote by $S_{u,v}(P)$ the SLP constructed in the proof of 
Proposition \ref{pr:Suv-SLP} for $P$.

We say that $P$ is \emph{reduced} if its label $\val(P)$ is reduced
as an element of $F(E)$. To reduce $P$ means to find
an SLP $P'$ satisfying the following conditions:
\begin{itemize}
\item 
$\val(P')$ is reduced as an element of $F(E)$,
\item 
$\val(P') =_{F(E)} \val(P)$,
\item 
$o(P')=o(P)$ or $\val(P') = \varepsilon$.
\end{itemize}

\begin{theorem}[{{\cite{Lohrey:2004}, Theorem 4.5}}]\label{th:reducing-SLP}
It takes polynomial-time to reduce $P$.
\end{theorem}

\section{The case of equal associated normal subgroups of finite index}
\label{se:case1}

In this section we consider the HNN extension $F \ast_\varphi t$
of the free group $F$ with equal associated subgroups $H$,
where $H$ is normal in $F$ and of finite index, and where the extension
is defined by an automorphism $\varphi \colon H \to H$.
Let $\Gamma=(V,E^\pm,\mu,r)$ be the subgroup graph for $H$.
These assumptions imply that
\begin{itemize}
\item 
$\Gamma$ is finite.
\item 
$\Gamma$ is $X$-regular.
\item 
For any $u,v\in V$ there is an automorphism $\varphi_{u,v}:\Gamma\to\Gamma$ 
satisfying $\varphi_{u,v}(u) = v$.
\end{itemize}
Since the group $G$ is fixed, we treat the following data as part 
of its description.
\begin{itemize}
\item 
The subgroup graph $\Gamma=(V,E^\pm,r,\mu)$ for $H$.
\item 
A set of edges $T\subseteq E^+$ defining a
spanning tree in $\Gamma$ as described in Section \ref{se:sbgp-basis}.
\item 
For every $e\in E^+\setminus T$ we have
\begin{itemize}
\item 
the circuit $p_e$ in $\Gamma$ corresponding to $e$ defined by \eqref{eq:p_e};
\item 
the circuit $p_e^\varphi$ in $\Gamma$ satisfying 
$\mu(p_e^\varphi) = \varphi(\mu(p_e))$;
\item 
the circuit $p_e^{\varphi^{-1}}$ in $\Gamma$ satisfying
$\mu(p_e^{\varphi^{-1}}) = \varphi^{-1}(\mu(p_e))$;
\item 
an SLP $P_e^\varphi$ satisfying
$\val(P_e^\varphi) = p_e^\varphi$;
\item 
an SLP $P_{e^{-1}}^\varphi$ satisfying
$\val(P_{e^{-1}}^\varphi) = (p_e^\varphi)^{-1}$;
\item 
an SLP $P_e^{\varphi^{-1}}$ satisfying
$\val(P_e^{\varphi^{-1}}) = p_e^{\varphi^{-1}}$;
\item 
an SLP $P_{e^{-1}}^{\varphi^{-1}}$ satisfying
$\val(P_{e^{-1}}^{\varphi^{-1}}) = (p_e^{\varphi^{-1}})^{-1}$.
\end{itemize}
\end{itemize}
Define a constant
\begin{equation}\label{eq:C}
C = C_\varphi =
\sum_{e\in E^+\setminus T}
|P_e^\varphi|+
|P_{e^{-1}}^\varphi|+
|P_e^{\varphi^{-1}}|+
|P_{e^{-1}}^{\varphi^{-1}}| +2.
\end{equation}

Now we describe the algorithm for the word problem in $F\ast_\varphi t$.
First, a given word  \eqref{eq:w} is translated into an alternating sequence
\begin{equation}\label{eq:slp-word}
P_0, t^{\varepsilon_1},
P_1, \dots,
P_{k-1}, t^{\varepsilon_k},P_k
\end{equation}
of letters $t^{\pm 1}$ and straight-line programs $P_i$
over the alphabet $E^\pm$
(a formal alphabet of edges of $\Gamma$)
satisfying the following conditions:
\begin{itemize}
\item 
$\val(P_i)$ is a path in $\Gamma$ starting from $r$,
\item 
$\mu(\val(P_i))=w_i$,
\end{itemize}
using Lemma \ref{le:CP-w}.
By Lemma \ref{le:CP-w}, \eqref{eq:slp-word}
can be computed in linear time.
All further computations are performed on the sequence \eqref{eq:slp-word}.

\subsection{Application of $\varphi^{\pm1}$ to an SLP}

\begin{proposition}\label{pr:apply-phi}
Let $P$ be an SLP over $\Gamma$ such that $\val(P)$ is a 
circuit in $\Gamma$ from $r$ to $r$.
Then $\mu(\val(P)) \in H$ and $\varphi$ is applicable to
$\mu(\val(P))$. 
There is an algorithm that
for a given $P$ produces an SLP $P'$ in $O(|P|)$ time
satisfying the following conditions:
\begin{itemize}
\item[(a)]
$\val(P')$ is a circuit based at $r$;
\item[(b)]
$\mu(\val(P')) = \varphi(\mu(\val(P)))$;
\item[(c)]
$|P'|\le |P| + \sum_{e\in E^+\setminus T}
|P_e^\varphi|+
|P_{e^{-1}}^\varphi|$.
\end{itemize}
The same holds for $\varphi^{-1}$.
\end{proposition}
    
\begin{proof}
The algorithm modifies the terminals $e\in E^\pm$ in $P$.
It distinguishes two types of terminals.

\textsc{(Case-I)}
For each terminal $e \in T$ and each non-terminal $N$ such that
$\delta(N) = e$ or $\delta(N) = e^{-1}$, the algorithm redefines
$\delta(N)$ to be $\varepsilon$.
This effectively deletes all occurrences of $e$ and $e^{-1}$
from $\val(P)$.

\textsc{(Case-II)}  
For each terminal $e\in E^+\setminus T$ the algorithm performs 
the following steps:
\begin{itemize}
\item 
Add the description of $P_e^{\varphi}$ and $P_{e^{-1}}^{\varphi}$ 
to $P$.
\item 
For each non-terminal $N$ such that $\delta(N)=e$,
redefine $\delta(N)$ to be $\rootSLP(P_e^{\varphi})$
\item 
For each non-terminal $N$ such that $\delta(N)=e^{-1}$,
redefine $\delta(N)$ to be $\rootSLP(P_{e^{-1}}^{\varphi})$.
\end{itemize}
This effectively replaces each occurrence of $e$ and $e^{-1}$
in $\val(P)$ with $\val(P_{e}^\varphi)$ 
and $\val(P_{e^{-1}}^\varphi)$, respectively.
Finally, the resulting SLP is pruned to remove all nonessential 
non-terminals.
By construction, the obtained SLP $P'$ satisfies all three properties.
\end{proof}

\subsection{The word problem algorithm}
\label{se:WP-normal-fi}

\begin{proposition}\label{pr:main-step-case-I}
Consider a segment $P_{i-1},t^{-1},P_i,t,P_{i+1}$ in \eqref{eq:slp-word}.
It requires $O(|P_{i-1}|+|P_{i}|+|P_{i+1}|)$ time to check if
$P_i$ defines an element in $H$ 
(i.e., if $\mu(\val(P_i)) \in H$)
and, if so, to construct
an SLP $P$ satisfying the following properties:
\begin{itemize}
\item[(a)]
$\val(P)$ is a path in $\Gamma$ that starts at $r$,
\item[(b)]
$\mu(\val(P))\ =_G\ 
\mu(\val(P_{i-1})) 
\cdot
t^{-1} \mu(\val(P_{i})) t
\cdot
\mu(\val(P_{i+1}))$,
\item[(c)]
$|P| \le |P_{i-1}| + |P_i| + |P_{i+1}| + C$.
\end{itemize}
The same holds for segments $P_{i-1},t,P_i,t^{-1},P_{i+1}$, when
$P_i$ defines an element in $H$.
\end{proposition}

\begin{proof}
By Lemmas \ref{le:first-last} and \ref{le:SLP-connected},
we can check if $\val(P_i)$ defines a path that
starts and ends at $r$ (i.e., if $\mu(\val(P_i)) \in H$)
in $O(|P_i|)$ time.
By Proposition \ref{pr:apply-phi}, 
in $O(|P_i|)$ time we can compute an SLP $P_i'$ satisfying
\begin{itemize}
\item 
$\val(P_i')$ a circuit in $\Gamma$ from $r$ to $r$, and
\item 
$\mu(\val(P_i')) =_G t^{-1} \mu(\val(P_i)) t$.
\end{itemize}
Hence, the word
$$
\mu(\val(P_{i-1}))\mu(\val(P_{i}'))\mu(\val(P_{i+1}))
$$
defines the same element as the right-hand side of (b).
Note that, in general,
$\val(P_{i-1}) \circ \val(P_{i}') \circ \val(P_{i+1}))$
does not define a continuous path in $\Gamma$; 
there may be up to two points of discontinuity.

To create a required SLP $P$, it remains to 
properly concatenate the paths
$\val(P_{i-1})$, $\val(P_{i}')$, $\val(P_{i+1})$,
which can be done using shift operations.
Use Lemma \ref{le:first-last} and Proposition \ref{pr:Suv-SLP}
to compute 
$$
v = t(\last(P_{i-1})),
\ P_i''=S_{r,v}(P_i')
\ \mbox{ and }\ 
P_{i+1}'=S_{r,v}(P_{i+1}).
$$
Observe that $\val(P_i'')$ is a circuit and $t(\last(P_{i}'')) = v$.
Hence, by definition of $P_i''$ and $P_{i+1}'$,
the concatenation
$\val(P_{i-1}) \circ \val(P_i'') \circ \val(P_{i+1}')$
is a path in $\Gamma$ that starts at $r$.
Its label defines the same element as the right-hand side of (b)
because shift operations preserve labels.
Therefore, the SLP $P = P_{i-1} \circ P_i''\circ P_{i+1}'$
satisfies (a) and (b).

Finally, by construction,
$$
|P_i''| = |P_i'| \le |P_i| + \sum_{e\in E^+\setminus T}
|P_e^\varphi|+
|P_{e^{-1}}^\varphi|
\ \mbox{ and }\ 
|P_{i+1}'| = |P_{i+1}|
$$
because applying shift operations does not change the size of an SLP.
Thus, 
$$
|P| \le 
|P_{i-1}|+|P_{i}| + |P_{i+1}|+
\sum_{e\in E^+\setminus T}
|P_e^\varphi|+
|P_{e^{-1}}^\varphi| +2
\le |P_{i-1}| + |P_{i}| + |P_{i+1}| + C
$$
and $P$ satisfies (c).
\end{proof}

\begin{theorem}\label{th:main-case1}
Suppose that $H$ is a normal subgroup of $F$ of finite index
and let $\varphi:H\to H$ be an automorphism.
Then the word problem for the HNN extension $F\ast_\varphi t$
is decidable in polynomial time.
\end{theorem}
    
\begin{proof}
Consider a word $w$ of type \eqref{eq:w}.
If the syllable length $k$ of $w$ is zero, then 
we directly check if $w_0$ is trivial in the base group $F$.

Suppose that $k\ge 1$.
Translate $w$ into a sequence \eqref{eq:slp-word}, which can be done 
in $O(|w|)$ time.
Then apply a sequence of Britton's reductions 
using Proposition \ref{pr:main-step-case-I} for a single reduction step.
If $w=_G1$, then the process produces a single SLP $P^\ast$
satisfying 
$$
|P^\ast| \le \sum_{i=0}^k |P_i| + \tfrac{k}{2} C
\le
(C+1)\sum_{i=0}^k |P_i|
$$
which is $O(|w|)$ because $C$ is a fixed parameter of the group.
The time complexity of reducing \eqref{eq:slp-word}
to $P^\ast$ can be bounded by $O(k|w|)$
or simply $O(|w|^2)$.
Finally, it remains to check if $\mu(\val(P^\ast)) = \varepsilon$ in $F$.
By Theorem \ref{th:reducing-SLP}, that can be done in polynomial time.
Therefore, the total time-complexity of the described procedure 
can be bounded by a polynomial.
\end{proof}

\section{The case of equal associated subgroups of finite index}
\label{se:normalizable}

In this section, we generalize the algorithm from Section~\ref{se:WP-normal-fi} 
to the case when $\varphi:H\to H$ can be restricted to a normal subgroup 
of $F$ of finite index, in which case we say that $\varphi$ is \emph{normalizable}.

\subsection{$\varphi$-invariant subgroups}

Let $H\le F$ and $\varphi \in\Aut(H)$.
We say that $H'\le H$ is \emph{$\varphi$-invariant} if 
$\varphi(H')= H'$, i.e., if $\varphi|_{H'} \in \Aut(H')$.

\begin{lemma}\label{le:phi_image}
$H'\le H$ is \emph{$\varphi$-invariant} if and only if 
$\varphi(H')\le H'$ and $\varphi^{-1}(H')\le H'$.
\end{lemma}

\begin{proof}
Applying $\varphi$ to $\varphi^{-1}(H') \le H'$, 
we obtain $H' \le \varphi(H')$. Together with the assumption 
$\varphi(H') \le H'$ this implies that $H' = \varphi(H')$.
\end{proof}

\begin{lemma}\label{le:join-meet}
If 
$\varphi(H')\le H'$ and
$\varphi(H'')\le H''$, then 
\begin{itemize}
\item[(a)]
$\varphi(\gp{H' \cup H''})\le \gp{H'\cup H''}$, and
\item[(b)]
$\varphi(H' \cap H'')\le H'\cap H''$.
\end{itemize}
\end{lemma}

\begin{proof}
Indeed, for any $g\in F$
\begin{align*}
g\in H'\cup H''
&\ \Rightarrow\ 
\varphi(g) \in \varphi(H') 
\ \mbox{ or }\ 
\varphi(g) \in \varphi(H'') \\
&\ \Rightarrow\ 
\varphi(g) \in \varphi(H') \cup \varphi(H'') \subseteq H'\cup H'',\\
g\in H'\cap H''
&\ \Rightarrow\ 
\varphi(g) \in \varphi(H') 
\ \mbox{ and }\ 
\varphi(g) \in \varphi(H'') \\
&\ \Rightarrow\ 
\varphi(g) \in \varphi(H') \cap \varphi(H'') \le H'\cap H'',
\end{align*}
which implies (a) and (b).
\end{proof}

\begin{lemma}[Join and meet operations]
If $H',H''$  are $\varphi$-invariant, then 
\begin{itemize}
\item[(a)]
$\gp{H'\cup H''}$ is $\varphi$-invariant, and
\item[(b)]
$H'\cap H''$ is $\varphi$-invariant.
\end{itemize}
\end{lemma}

\begin{proof}
Follows from Lemma \ref{le:phi_image} and 
Lemma \ref{le:join-meet} applied 
to $\varphi$ and $\varphi^{-1}$.
\end{proof}

Thus, the set of all $\varphi$-invariant subgroups, denoted by $L_\varphi$,
has a structure of a \emph{bounded lattice},
with the maximum element $H$ and the minimum element $\{1\}$.
Let us consider the set
$$
L_\varphi^\ast = \Set{N\in L_\varphi}{N\unlhd F}.
$$
$L_\varphi^\ast$ is not empty, as it contains $\{1\}$.
Furthermore, it is easy to check that it is a sublattice of $L_\varphi$.
Denote by $M_\varphi$ the maximum element of $L_\varphi^\ast$.

\subsection{$\varphi$-invariant normal subgroups of finite index}
\label{se:subsectionnormalizable}

Now, let $H\le F$ be a subgroup of finite index and $\varphi\in\Aut(H)$.
We say that $\varphi$ is \emph{normalizable}
if $\varphi$ can be restricted
to a normal subgroup $N\unlhd F$ of finite index
(i.e., if $M_\varphi$ has finite index).

\begin{problem}
Is it true that every $\varphi$ is normalizable?
\end{problem}

We suspect that the answer is negative in general.
However, how can one find such a subgroup $N$ if it exists?
Let us review the properties of a required subgroup $N$.
It should satisfy the following three properties:
\begin{itemize}
\item 
$N$ is normal,
\item 
$N$ has finite index,
\item 
$\varphi^{\pm 1}(N) \leq N$.
\end{itemize}
Let us define the following sequence of subgroups:
\begin{equation}\label{eq:Hi-seq}
\begin{array}{l}
H_0=H,\\
H_{i+1}=
\bigcap_c c^{-1} H_i c 
\ \cap\ \varphi(H_i) 
\ \cap\ \varphi^{-1}(H_i)
\ \ \ \mbox{ for }i\ge 0.
\end{array}
\end{equation}

\begin{lemma}\label{le:H-k-finiteindex}
$[F:H_k]<\infty$ for every $k\ge 0$.
\end{lemma}

\begin{proof}
Induction on $k$.
By assumption, $H_0=H$ has finite index.
Since, $\varphi^{\pm 1}(H) = H$ we have
$$H_{1}=\bigcap_c c^{-1} H_0 c$$
which has finite index in $F$.
Assume that the statement holds for $H_{k-1}$ and consider $H_k$.
\begin{align*}
[F:H_{k-1}]<\infty&\ \ \Rightarrow\ \ [H:H_{k-1}]<\infty\\
&\ \ \Rightarrow\ \ [H:\varphi(H_{k-1})]<\infty
&&\mbox{(because $\varphi\in\Aut(H)$)} \\
&\ \ \Rightarrow\ \ [F:\varphi(H_{k-1})]<\infty.
\end{align*}
Similarly $[F:\varphi^{-1}(H_{k-1})]<\infty$.
Hence, $H_k$ is an intersection of finitely many subgroups
of finite index and, hence, has finite index itself.
\end{proof}

\begin{lemma}\label{le:M-Hi}
$M_\varphi\le H_k$ for every $k\ge 0$.
\end{lemma}

\begin{proof}
Induction on $k$. The statement holds for $k=0$.
Assume that the statement holds for $k-1$, i.e., $M_\varphi\le H_{k-1}$.
Then the following holds.
\begin{align*}
M_\varphi &= \cap_c c^{-1} M_\varphi c\le \cap_c c^{-1} H_{k-1} c. \\
M_\varphi &=\varphi(M_\varphi)\le \varphi(H_{k-1}).\\
M_\varphi &=\varphi^{-1}(M_\varphi)\le \varphi^{-1}(H_{k-1}).
\end{align*}
Therefore,
$M_\varphi\le 
\cap_c c^{-1} H_{k-1} c \cap \varphi(H_{k-1}) \cap \varphi^{-1}(H_{k-1}) 
= H_k$.
\end{proof}

\begin{lemma}
$\cap_i H_i = M_\varphi$.
\end{lemma}

\begin{proof}
Denote $\cap_i H_i$ by $L$. It is easy to see that it is 
a normal subgroup of $F$. Moreover,
\begin{align*}
x\in L
&\ \ \Rightarrow\ \ x\in H_i\le \varphi^{-1}(H_{i-1})\ \forall i\ge 1\\
&\ \ \Rightarrow\ \ \varphi(x)\in H_{i-1} \ \forall i\ge 1\\
&\ \ \Rightarrow\ \ \varphi(x)\in L.
\end{align*}
Hence, $\varphi(L) \le L$ and, similarly, $\varphi^{-1}(L) \le L$.
Therefore, $L$ is $\varphi$-invariant and $L \in L_\varphi^\ast$.
By Lemma~\ref{le:M-Hi}, we have $M_\varphi \le L$.
Thus, $L = M_\varphi$.
\end{proof}

We say that the sequence \eqref{eq:Hi-seq} \emph{stabilizes} 
if $H_{i+1}=H_i$ for some index $i$.
The next lemma follows from the definition of $H_{i+1}$.

\begin{lemma}\label{le:Hi-stabilizes}
If $H_{i+1}=H_i$, then $H_i=M_\varphi$.
\end{lemma}

\begin{proposition}
$\varphi$ is normalizable if and only if 
the sequence $H_0\ge H_1\ge H_2\ge \dots$ 
stabilizes.
\end{proposition}

\begin{proof}
The direction ``$\Leftarrow$'' follows directly from
Lemmas~\ref{le:H-k-finiteindex} and~\ref{le:Hi-stabilizes}.

To prove the converse, assume that $\varphi$ is normalizable.
Then $[F : M_\varphi] = m < \infty$ and
\[
H_{\lceil \log_2 m \rceil} = M_\varphi,
\]
since $[H_i : H_{i+1}] \ge 2$ whenever $H_{i+1} \neq H_i$.
\end{proof}

\subsection{The case when $\varphi$ extends to an automorphism of $F$}

Recall that a subgroup  $N\le F$ is \emph{characteristic} if 
$\varphi(N)\leq N$ for every $\varphi \in \Aut(F)$. 
It is easy to see that every characteristic subgroup is normal.

\begin{lemma}\label{le:characteristic}
Let $H$ be a finite index subgroup of $F$.
Then $H$ contains a characteristic subgroup of $F$ of finite index.
\end{lemma}

\begin{proof}
Let $[F:H] = n <\infty$. Then $[F:\varphi(H)] = n$ for any $\varphi\in\Aut(F)$. We know that for a finitely generated group, it has a finite number of subgroups of index $n$
 for any $n\in\MN$.
The number of subgroups of index $n$ in $F$ is finite.
Therefore
$$
N = \bigcap_{\varphi\in\Aut(F)} \varphi(H)
$$
has finite index, is characteristic, and is contained in $H$.
\end{proof}

\begin{proposition}
If $\varphi\in\Aut(H)$ extends to an automorphism of $F$, 
then $\varphi$ is normalizable.
\end{proposition}

\begin{proof}
By Lemma \ref{le:characteristic}, $H$ contains 
a finite index characteristic subgroup $N$ of $F$.
\end{proof}

\subsection{$\WP(F\ast_\varphi t)$ is decidable in polynomial time when $\varphi$ is normalizable}

Suppose that $\varphi$ is normalizable and $N$ is a normal subgroup of $F$
of finite index satisfying $\varphi|_N \in \Aut(N)$. Denote $\varphi|_N$
by $\varphi^\ast$.

\subsubsection{Syllables}

Let $\Gamma^\ast=(V,E^\pm,r,\mu)$ be the subgroup graph for $N$.
The graph $\Gamma^\ast$ is the Cayley graph of the finite group $F/N$,
which induces a natural group operation on its vertex set $V$.
Moreover, since $N\le H\le F$, it follows that $H/N\le F/N$.
Hence, for $v\in V$ we write $v\in H$ 
whenever the coset corresponding to $v$ belongs to $H/N$.
In particular, we have $r\in H$.

As in Section~\ref{se:sbgp-basis}, fix a spanning tree
$T\subseteq E^+$ in $\Gamma^\ast$.
As in Section~\ref{se:case1}, we work with paths over $\Gamma^\ast$.
In this case, however, each path is represented by a pair $(P, v)$, 
where $P$ is an SLP over the alphabet $E^\pm$ and $v \in V$,
subject to the following condition:
\begin{itemize}
\item 
$\val(P)$ is a circuit in $\Gamma^\ast$ based at $r$.
\end{itemize}
We call such a pair a \emph{syllable}. Each syllable defines the path 
$$
p(P,v) = \val(P) [r,v]_T
$$
in $\Gamma^\ast$ that starts at $r$, ends at $v$, and which label is
$$
w(P,v) = \mu(\val(P)) \mu([r,v]_T).
$$

\subsubsection{Precomputed data}

Since the group $G=F\ast_\varphi t$ is fixed,
we assume that the following data is included in its description.
\begin{itemize}
\item 
The subgroup graph $\Gamma$ for $H$.
\item 
The subgroup graph $\Gamma^\ast=(V,E^\pm,r,\mu)$ for $N$.
\item 
A set of edges $T\subseteq E^+$ defining a
spanning tree in $\Gamma$ as described in Section \ref{se:sbgp-basis}.
\item 
The index $k=[F:N]$ (the same as the order of $\Gamma^\ast$).
\item 
For every $e\in E^+\setminus T$ we have
\begin{itemize}
\item 
the circuit $p_e$ in $\Gamma^\ast$ corresponding to $e$ defined by \eqref{eq:p_e};
\item 
the circuit $p_e^\varphi$ in $\Gamma^\ast$ satisfying 
$\mu(p_e^\varphi) = \varphi(\mu(p_e))$;
\item 
the circuit $p_e^{\varphi^{-1}}$ in $\Gamma^\ast$ satisfying
$\mu(p_e^{\varphi^{-1}}) = \varphi^{-1}(\mu(p_e))$;
\item 
an SLP $P_e^\varphi$ satisfying
$\val(P_e^\varphi) = p_e^\varphi$;
\item 
an SLP $P_{e^{-1}}^\varphi$ satisfying
$\val(P_{e^{-1}}^\varphi) = (p_e^\varphi)^{-1}$;
\item 
an SLP $P_e^{\varphi^{-1}}$ satisfying
$\val(P_e^{\varphi^{-1}}) = p_e^{\varphi^{-1}}$;
\item 
an SLP $P_{e^{-1}}^{\varphi^{-1}}$ satisfying
$\val(P_{e^{-1}}^{\varphi^{-1}}) = (p_e^{\varphi^{-1}})^{-1}$.
\end{itemize}
\item 
For every vertex $v$ such that $v\in H$ we have 
a syllable $(P_v^\varphi,v')$ satisfying
$$
w(P_v^\varphi,v') = \varphi(\mu([r,v]_T)),
$$
and a similar pair $(P_v^{\varphi^{-1}},v')$ for $\varphi^{-1}$.
\item
For every $v_1,v_2\in V$ we have an SLP $T_{v_1,v_2}$ constructed
by Lemma \ref{le:CP-w} satisfying 
$\val(T_{v_1,v_2}) = [v_1,v_2]_T$.
\item 
For every $v_1,v_2\in V$ we have an SLP
$C_{v_1,v_2} = S_{r,v_1}(T_{r,v_2}) \circ T_{v_1v_2,r}$, where
$v_1v_2$ in $T_{v_1v_2,r}$ is the product of $v_1$ and $v_2$ in $F/N$.
\end{itemize}
Define a constant 
\begin{align*}
C^\ast = &
\sum_{e\in E^+\setminus T}
(|P_e^\varphi|+
|P_{e^{-1}}^\varphi|+
|P_e^{\varphi^{-1}}|+
|P_{e^{-1}}^{\varphi^{-1}}|) \\
&+\sum_{v\in V} (|P_v^\varphi|+|P_v^{\varphi^{-1}}|) \\
&+2\sum_{v\in V} |T_{r,v}| + 2\sum_{v_1,v_2\in V} |C_{v_1,v_2}| + 7.
\end{align*}

\subsubsection{Data representation for a given word $w$}

A given word  \eqref{eq:w} is translated into an alternating sequence
\begin{equation}\label{eq:slp-word2}
(P_0,v_{0}), t^{\varepsilon_1},
(P_1,v_{1}), \dots,
(P_{k-1},v_{k-1}), t^{\varepsilon_k},
(P_k,v_{k})
\end{equation}
of letters $t^{\pm 1}$ and syllables $(P_i,v_i)$ satisfying
$
w(P_i,v_i) = w_i.
$
The sequence \eqref{eq:slp-word2} is computed by applying 
Lemma \ref{le:CP-w2} to $w_0,\dots,w_k$.

\begin{lemma}\label{le:CP-w2}
For a given $w\in F$, it requires $O(|w|)$ time to construct
a syllable $(P,v)$ satisfying $w(P,v) = w$.
\end{lemma}

\begin{proof}
Let $v$ be the endpoint of the path labeled by $w$ in $\Gamma^\ast$ 
starting from $r$.
The word $w\cdot \mu([v,r]_T)$ labels a circuit $c$ in $\Gamma^\ast$,
and its length is bounded by $|w|+|\Gamma^\ast|$.
Using Lemma \ref{le:CP-w}, construct an SLP $P$ such that $\val(P)=c$.
The pair $(P,v)$ is a required syllable.
\end{proof}

\subsubsection{Application of $\varphi^{\pm1}$ to $(P,v)$}

Consider a syllable $(P,v)$. Obviously,
$$
w(P,v)\in H\ \ \Leftrightarrow\ \ v\in H.
$$
To apply $\varphi^\pm$ to $(P,v)$ means to compute a syllable $(P',v')$
satisfying $w(P',v')=\varphi^\pm(w(P,v))$.

\begin{proposition}\label{pr:apply-phi2}
There is an algorithm that
for a given syllable $(P,v)$, that satisfies $w(P,v)\in H$,
produces a syllable $(P',v')$ in $O(|P|)$ time
satisfying the following conditions:
\begin{itemize}
\item[(a)]
$w(P',v')=\varphi(w(P,v))$,
\item[(b)]
$|P'|\le |P| + \sum_{e\in E^+\setminus T}
(|P_e^\varphi|+
|P_{e^{-1}}^\varphi|) 
+|P_v^\varphi|+1$.
\end{itemize}
A similar statement holds for $\varphi^{-1}$.
\end{proposition}

\begin{proof}
Since $\mu(\val(P))\in N$, we can 
process $P$ using Proposition \ref{pr:apply-phi} and
denote the result by $P_1$.
That increases
the number of non-terminals by at most 
$\sum_{e\in E^+\setminus T}
(|P_e^\varphi|+ |P_{e^{-1}}^\varphi|)$.
Then use the precomputed pair $(P_v^{\varphi},v')$ 
defining $\varphi(\mu([r,v]_T))$ and
concatenate $P_1$ and $P_v^{\varphi}$ to get $P'$.
That adds $|P_v^\varphi|+1$ non-terminals.
Clearly, $(P',v')$ is a required syllable.
\end{proof}

\subsubsection{Syllable concatenation}

To \emph{concatenate} syllables $(P_1,v_1)$ and $(P_2,v_2)$ means
to construct a syllable $(P,v)$ satisfying
\begin{equation}\label{eq:syllable-concat-prop}
w(P_1,v_1)w(P_2,v_2) = w(P,v).
\end{equation}
One way to construct a required pair $(P,v)$ is to 
shift $p(P_2,v_2)$ in $\Gamma^\ast$ so that its origin is $v_1$
and attach the result to $(P_1,v_1)$,
see Figure \ref{fi:syllable-concatenation}, and then construct 
a proper syllable.
\begin{figure}[h]
\centering
\includegraphics[scale=0.8]{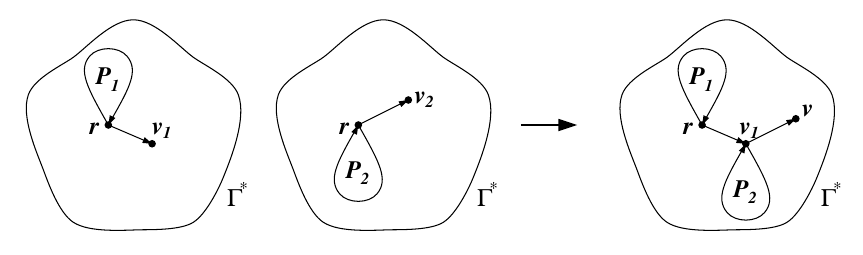}
\caption{Concatenation of the paths $p(P_1,v_1)$ and $p(P_2,v_2)$.}
\label{fi:syllable-concatenation}
\end{figure}
In the next proposition we prove that the syllable 
\begin{equation}\label{eq:syllable-concat}
(P,v) = (P_1 \circ T_{r,v_1} \circ S_{r,v_1}(P_2) \circ 
\underbrace{S_{r,v_1}(T_{r,v_2}) \circ T_{v_1v_2,r}}_{C_{v_1,v_2}},
v_1v_2)
\end{equation}
defines concatenation of $(P_1,v_1)$ and $(P_2,v_2)$.
We denote $(P,v)$ defined in \eqref{eq:syllable-concat} by
$(P_1,v_1) \circ (P_2,v_2)$.

\begin{proposition}\label{pr:syllable-concat}
$(P,v) = (P_1,v_1) \circ (P_2,v_2)$
satisfies \eqref{eq:syllable-concat-prop}.
Constructing $(P,v)$ requires $O(|P_1|+|P_2|)$ time. Moreover,
$$
|P|\le |P_1| + |P_2| + |T_{r,v_1}| + |C_{v_1,v_2}| + 3.
$$
\end{proposition}

\begin{proof}
By definition,
\begin{align*}
w(P_1,v_1) = \mu(\val(P_1)) \mu([r,v_1]_T) 
&= 
\mu(\val(P_1)) \mu(\val(T_{r,v_1})),\\
w(P_2,v_2) = \mu(\val(P_2)) \mu([r,v_2]_T) 
&= \mu(\val(P_2)) \mu(\val(T_{r,v_2}))\\
&= \mu(\val(S_{r,v_1}(P_2))) \mu(S_{r,v_1}(T_{r,v_2})),
\end{align*}
because an application of $S_{r,v_1}$ does not change the labels.
Hence, 
$$
w(P_1,v_1)w(P_2,v_2) =
\mu(\val(P_1 \circ T_{r,v_1} \circ S_{r,v_1}(P_2) \circ S_{r,v_1}(T_{r,v_2}))).
$$
Notice that 
$\val(P_1 \circ T_{r,v_1} \circ S_{r,v_1}(P_2) \circ S_{r,v_1}(T_{r,v_2}))$
is a path
in $\Gamma^\ast$ from $r$ to $v_1v_2$.
Since the path $\val(T_{v_1v_2,r} \circ T_{r,v_1v_2})$
freely reduces to $\varepsilon$ as an element of $F(E)$, we have
$$
w(P_1,v_1)w(P_2,v_2) =
\mu(\val(P_1 \circ 
T_{r,v_1} \circ
S_{r,v_1}(P_2) \circ
S_{r,v_1}(T_{r,v_2}) \circ
T_{v_1v_2,r} \circ
T_{r,v_1v_2})),
$$
where $\val(P_1 \circ 
T_{r,v_1} \circ
S_{r,v_1}(P_2) \circ
S_{r,v_1}(T_{r,v_2}) \circ
T_{v_1v_2,r})$ 
is a circuit based at $r$ and 
$\val(T_{r,v_1v_2}) = [r,v_1v_2]_T$.
Thus, $(P_1 \circ 
T_{r,v_1} \circ
S_{r,v_1}(P_2) \circ
S_{r,v_1}(T_{r,v_2}) \circ
T_{v_1v_2,r},v_1v_2)$ satisfies \eqref{eq:syllable-concat-prop}.

We have $P = P_1 \circ T_{r,v_1} \circ S_{r,v_1}(P_2) \circ C_{v_1,v_2}$,
by definition of $C_{v_1,v_2}$.
By Proposition \ref{pr:Suv-SLP}, $S_{r,v_1}(P_2)$
can be computed in $O(|P_2|)$ time and 
satisfies $|S_{r,v_1}(P_2)| = |P_2|$.
Then concatenation 
$P_1 \circ T_{r,v_1} \circ S_{r,v_1}(P_2) \circ C_{v_1,v_2}$
can be computed in $O(|P_1|+|P_2|)$ time because
$T_{r,v_1}$ and $C_{v_1,v_2}$ are of constant size.
The size of concatenation is bounded by the sum of sizes
$|P_1| + |P_2| + |T_{r,v_1}| + |C_{v_1,v_2}|$ plus three
additional non-terminals.
\end{proof}

\subsubsection{The algorithm}

\begin{proposition}\label{pr:main-step-case-II}
Consider a segment $(P_{i-1},v_{i-1}),t^{-1},(P_i,v_i),t,(P_{i+1},v_{i+1})$ 
in \eqref{eq:slp-word2}.
There is an algorithm that in $O(|P_{i-1}|+|P_{i}|+|P_{i+1}|)$ time
verifies if $w(P_i,v_i) \in H$ and, if so, constructs
a syllable $(P,v)$ satisfying the following properties:
\begin{itemize}
\item[(a)]
$w(P,v) =_G
w(P_{i-1},v_{i-1})
\cdot
t^{-1} w(P_{i},v_{i}) t
\cdot
w(P_{i+1},v_{i+1})$,
\item[(b)]
$|P| \le |P_{i-1}| + |P_i| + |P_{i+1}| + C^\ast$.
\end{itemize}
The same holds for segments 
$(P_{i-1},v_{i-1}),t,(P_i,v_i),t^{-1},(P_{i+1},v_{i+1})$.
\end{proposition}

\begin{proof}
To verify if $w(P_i,v_i) \in H$ it is sufficient to check if $v_i$, 
viewed as an element of $F/N$, belongs to $H/N$.
This can be precomputed for $\Gamma^\ast$, and hence can be checked in 
$O(1)$ time. Suppose that it is the case.

Using Proposition \ref{pr:apply-phi2}, one can compute in $O(|P_i|)$ time
a syllable $(P',v')$ satisfying
$w(P',v') =_F \varphi(w(P_i,v_i)) =_G t^{-1} w(P_{i},v_{i}) t$.
Then, using Proposition \ref{pr:syllable-concat} twice, one can compute 
in $O(|P_{i-1}|+|P_{i}|+|P_{i+1}|)$ time
concatenation $(P,v)$ of 
three syllables $(P_{i-1},v_{i-1})$, $(P',v')$, 
and $(P_{i+1},v_{i+1})$.
Now, (a) follows from the construction of $(P,v)$.
Furthermore,
\begin{align*}
|P| \le& |P_{i-1}| + |P'| + |P_{i+1}| +
2\bigg(\sum_{v} |T_{r,v}| + \sum_{v_1,v_2} |C_{v_1,v_2}|+ 3\bigg)
&&\mbox{(Proposition \ref{pr:syllable-concat}(b))}\\
\le & |P_{i-1}| + |P_i| + |P_{i+1}| +
2\bigg(\sum_{v} |T_{r,v}| + \sum_{v_1,v_2} |C_{v_1,v_2}|+ 3\bigg)\\
&+\sum_{e\in E^+\setminus T} (|P_e^\varphi|+ |P_{e^{-1}}^\varphi|) 
+\sum_v |P_v^\varphi|+1
&&\mbox{(Proposition \ref{pr:apply-phi2}(b))}\\
\le& |P_{i-1}| + |P_i| + |P_{i+1}| + C^\ast
\end{align*}
and (b) holds.
\end{proof}



\begin{theorem}\label{th:main-case2}
Suppose that $H$ is a subgroup of $F$ of finite index 
and let $\varphi:H\to H$ be a normalizable isomorphism.
Then the word problem for the HNN extension $F\ast_\varphi t$
is decidable in polynomial time.
\end{theorem}

\begin{proof}
The proof is the same as the proof of Theorem \ref{th:main-case1},
but instead of Proposition \ref{pr:main-step-case-I}
we use Proposition \ref{pr:main-step-case-II} 
for a single Britton reduction step.
\end{proof}

\bibliography{main_bibliography}

\end{document}